\documentclass[12pt,british]{article}
\usepackage{ae,aecompl}

\usepackage[T1]{fontenc}
\usepackage[latin9]{inputenc}
\usepackage[a4paper]{geometry}
\usepackage[active]{srcltx}
\usepackage{color}
\usepackage[british]{babel}
\usepackage{array}
\usepackage{float}
\usepackage[figuresright]{rotating}
\usepackage{calc}
\usepackage{textcomp}
\usepackage{multirow}
\usepackage{algorithm2e}
\usepackage{amsmath}
\usepackage{amsthm}
\usepackage{amssymb}
\usepackage{mathptmx}
\usepackage{graphicx}
\usepackage{esint}
\usepackage[unicode=true,pdfusetitle,
 bookmarks=true,bookmarksnumbered=false,bookmarksopen=false,
 breaklinks=false,pdfborder={0 0 1},backref=false,colorlinks=true]
 {hyperref}
\hypersetup{
 citecolor=green,linkcolor=blue}
\usepackage{breakurl}
\usepackage{babel}
\usepackage{algpseudocode}
\usepackage{url}

\newtheorem{theorem}{Theorem}[section]
\theoremstyle{definition}
\newtheorem{definition}{Definition}[section]
\newtheorem{lemma}{Lemma}[section]
\newtheorem{proposition}{Proposition}[section]
\newtheorem{corollary}{Corollary}[section]

\theoremstyle{remark}
\newtheorem{remark}{Remark}[section]
\newtheorem{example}{Example}[section]

\makeatletter

\renewcommand{\labelenumi}{(\roman{enumi})}

\makeatother

\begin{document}

\title{Accelerating the DC algorithm for smooth functions}

\author{Francisco J. Aragón Artacho%
\thanks{Department of Mathematics, University of Alicante,
Spain. E-mail: \protect\protect\href{mailto:francisco.aragon@ua.es}{francisco.aragon@ua.es}%
} \and Ronan M.T. Fleming%
\thanks{Systems Biochemistry Group, Luxembourg Centre for Systems Biomedicine,
University of Luxembourg, Campus Belval, 4362 Esch-sur-Alzette, Luxembourg.
E-mail: \protect\protect\href{mailto:ronan.mt.fleming@gmail.com}{ronan.mt.fleming@gmail.com}%
} \and Phan T. Vuong %
\thanks{Systems Biochemistry Group, Luxembourg Centre for Systems Biomedicine,
University of Luxembourg, Campus Belval, 4362 Esch-sur-Alzette, Luxembourg.
E-mail: \protect\protect\href{mailto:vuongphantu@gmail.com}{vuongphantu@gmail.com}%
}}

\date{\today}

\maketitle
\begin{abstract}
We introduce two new algorithms to minimise smooth difference of convex~(DC)
functions that accelerate the convergence of the classical DC algorithm
(DCA). We prove that the point computed by DCA can be used to define
a descent direction for the objective function evaluated at this point.
Our algorithms are based on a combination of DCA together with a line
search step that uses this descent direction. Convergence of the algorithms
is proved and the rate of convergence is analysed under the \L{}ojasiewicz
property of the objective function. We apply our algorithms to a class
of smooth DC programs arising in the study of biochemical reaction
networks, where the objective function is real analytic and thus satisfies
the \L{}ojasiewicz property. Numerical tests on various biochemical
models clearly show that our algorithms outperforms DCA, being on
average more than four times faster in both computational time and
the number of iterations. Numerical experiments show that the algorithms are globally convergent to a non-equilibrium steady state of various biochemical networks, with only chemically consistent restrictions on the network topology.
\end{abstract}
%

\section{Introduction}

Many problems arising in science and engineering applications require
the development of algorithms to minimise a nonconvex function. If
a nonconvex function admits a decomposition, this may be exploited
to tailor specialised optimisation algorithms. Our main focus is
the following optimisation problem
\begin{equation}
\ensuremath{\underset{x\in\mathbb{R}^{m}}{\textrm{minimise}}}\;\phi(x):=f_{1}(x)-f_{2}(x),\label{eq:DC1}
\end{equation}
where $f_{1},f_{2}:\mathbb{R}^{m}\to\mathbb{R}$ are continuously
differentiable convex functions and
\begin{equation}
\inf_{x\in\mathbb{R}^{m}}\phi(x)>-\infty.\label{eq:phi_bounded_below}
\end{equation}
In our case, as we shall see in Section~4, this problem arises in
the study of biochemical reaction networks. In general, $\phi$ is
a nonconvex function. The function in problem~\eqref{eq:DC1} belongs
to two important classes of functions: the class of functions that
can be decomposed as a sum of a convex function and a differentiable
function (\emph{composite functions}) and the class of functions that
are representable as difference of convex functions (\emph{DC functions}).

In 1981, M. Fukushima and H. Mine \cite{fukushima_generalized_1981,mine_minimization_1981}
introduced two algorithms to minimise a \emph{composite function}.
In both algorithms, the main idea is to linearly approximate the differentiable
part of the composite function at the current point and then minimise
the resulting convex function to find a new point. The difference
between the new and current points provides a descent direction with
respect to the composite function, when it is evaluated at the current
point. The next iteration is then obtained through a line search procedure
along this descent direction. Algorithms for minimising composite
functions have been extensively investigated and found applications
to many problems such as: inverse covariance estimate, logistic regression,
sparse least squares and feasibility problems, see e.g.~\cite{huang2014barzilai,Lee2014a,Li2014,nesterov2013gradient}
and the references quoted therein.

In 1986, T. Pham Dinh and S. El Bernoussi \cite{TaoSouad1986} introduced
an algorithm to minimise \emph{DC functions}.  In its simplified
form, the \emph{Difference of Convex functions Algorithm} (DCA) linearly
approximates the concave part of the objective function ($-f_{2}$
in \eqref{eq:DC1}) about the current point and then minimises the
resulting convex approximation to the DC function to find the next
iteration, without recourse to a line search. The main idea is similar
to Fukushima--Mine approach but was extended to the non-differentiable
case. This algorithm has been extensively studied by H.A.~Le Thi,
T.~Pham Dinh and their collaborators, see e.g.~\cite{an_numerical_1996,le2009convergence,tao1998dc,tao2005dc}.
DCA has been successfully applied in many fields, such as machine
learning, financial optimisation, supply chain management and telecommunication
\cite{collobert2006trading,schnorr2007variational,tao2005dc}. Nowadays,
DC programming plays an important role in nonconvex programming and
DCA is commonly used because of its key advantages: simplicity, inexpensiveness
and efficiency \cite{tao2005dc}.  Some results related
to the convergence rate for special classes of DC programs have been
also established~\cite{le2011solving,le2009convergence}.

In this paper we introduce two new algorithms to find stationary points
of DC programs, called \emph{Boosted Difference of Convex function
Algorithms} (BDCA), which accelerate DCA with a line search using
an Armijo type rule. The first algorithm directly uses a backtracking
technique, while the second uses a quadratic interpolation of the
objective function together with backtracking. Our algorithms are
based on both DCA and the proximal point algorithm approach of Fukushima--Mine.
First, we compute the point generated by DCA. Then, we use this point
to define the search direction. This search direction coincides with
the one employed by Fukushima--Mine in~\cite{fukushima_generalized_1981}.
The key difference between their method and ours is the starting point
used for the line search: in our algorithms we use the point generated
by DCA, instead of using the previous iteration. This scheme works
thanks to the fact that the defined search direction is not only a
descent direction for the objective function at the previous iteration,
as observed by Fukushima--Mine, but is also a descent direction at
the point generated by DCA. Unfortunately, as shown in Remark~\ref{rem:nonsmooth},
this scheme cannot be extended in general for nonsmooth functions, as the defined search
direction might be an ascent direction at the point generated by DCA.

Moreover, it is important to notice that the iterations of Fukushima--Mine
and BDCA never coincide, as the largest step size taken in their algorithm
is equal to one (which gives the DCA iteration). In fact, for smooth
functions, the iterations of Fukushima--Mine usually coincide with
the ones generated by DCA, as the step size equal to one is normally
accepted by their Armijo rule.

We should point out that DCA is a descent method without line search.
This is something that is usually claimed to be advantageous in the
large-scale setting. Our purpose here is the opposite: we show that
a line search can increase the performance even for high-dimensional
problems.

Further, we analyse the rate of convergence under the \L{}ojasiewicz
property~\cite{lojasiewicz1965ensembles} of the objective function.
It should be mentioned that the \L{}ojasiewicz property is recently
playing an important role for proving the convergence of optimisation
algorithms for analytic cost functions, see e.g.~\cite{absil2005convergence,attouch2009convergence,bolte2007lojasiewicz,le2009convergence}.

We have performed numerical experiments in functions arising in the
study of biochemical reaction networks. We show that the problem of
finding a steady state of these networks, which plays a crucial role
in the modelling of biochemical reaction systems, can be reformulated
as a minimisation problem involving DC functions. In fact, this is
the main motivation and starting point of our work: when one applies
DCA to find a steady state of these systems, the rate of convergence
is usually quite slow. As these problems commonly involve hundreds
of variables (even thousands in the most complex systems, as Recon~2%
\footnote{Recon 2 is the most comprehensive representation of human metabolism
that is applicable to computational modelling~\cite{Thiele2013}.
This biochemical network model involves more than four thousand molecular
species and seven thousand reversible elementary reactions.%
}), the speed of convergence becomes crucial. In our numerical tests
we have compared BDCA and DCA for finding a steady state in various
biochemical network models of different size. On average, DCA needed
five times more iterations than BDCA to achieve the same accuracy,
and what is more relevant, our implementation of BDCA was more than
four times faster than DCA to achieve the same accuracy. Thus, we
prove both theoretically and numerically that BDCA results more advantageous
than DCA. Luckily, the objective function arising in\emph{ }these
biochemical reaction networks is real analytic, a class of functions
which is known to satisfy the \L{}ojasiewicz property~\cite{lojasiewicz1965ensembles}.
Therefore, the above mentioned convergence analysis results can be
applied in this setting.

The rest of this paper is organised as follows. In Section~\ref{sec:Preliminaries},
we recall some preliminary facts used throughout the paper and we
present the main optimisation problem. Section~\ref{sec:BDCA} describes
our main results, where the new algorithms (BDCA) and their convergence
analysis for solving DC programs are established. A DC program arising
in biochemical reaction network problems is introduced in Section~\ref{sec:A-DC-problem}.
Numerical results comparing BDCA and DCA on various biochemical network
models are reported in Section~\ref{sec:Application-to-biochemistry-1}.
Finally, conclusions are stated in the last section.

\section{Preliminaries\label{sec:Preliminaries}}

Throughout this paper, the inner product of two vectors $x,y\in\mathbb{R}^{m}$
is denoted by $\langle x,y\rangle$, while $\|\cdot\|$ denotes the
induced norm, defined by $\|x\|=\sqrt{\langle x,x\rangle}$. The nonnegative
orthant in~$\mathbb{R}^{m}$ is denoted by~$\mathbb{R}_{+}^{m}=[0,\infty)^{m}$
and $\mathbb{B}(x,r)$ denotes the closed ball of center $x$ and
radius $r>0$. The gradient of a differentiable function $f:\mathbb{R}^{m}\to\mathbb{R}^{n}$
at some point $x\in\mathbb{R}^{m}$ is denoted by $\nabla f(x)\in\mathbb{R}^{m\times n}$.

Recall that a function $f:\mathbb{R}^{m}\to\mathbb{R}$ is said to
be \emph{convex} if
\[
f(\lambda x+(1-\lambda)y)\leq\lambda f(x)+(1-\lambda)f(y)\quad\text{for all }x,y\in\mathbb{R}^{m}\text{ and }\lambda\in(0,1).
\]
Further, $f$ is called \emph{strongly convex} with modulus $\sigma>0$
if
\[
f(\lambda x+(1-\lambda)y)\leq\lambda f(x)+(1-\lambda)f(y)-\frac{1}{2}\sigma\lambda(1-\lambda)\|x-y\|^{2}\quad\text{for all }x,y\in\mathbb{R}^{m}\text{ and }\lambda\in(0,1),
\]
or, equivalently, when $f-\frac{\sigma}{2}\|\cdot\|^{2}$ is convex.
The function $f$ is said to be \emph{coercive} if ~$f(x)\to+\infty$
whenever $\left\Vert x\right\Vert \to+\infty.$

On the other hand, a function $F:\mathbb{R}^{m}\to\mathbb{R}^{m}$
is said to be \emph{monotone} when
\[
\langle F(x)-F(y),x-y\rangle\geq0\quad\text{for all }x,y\in\mathbb{R}^{m}.
\]
Further, $F$ is called \emph{strongly monotone} with modulus $\sigma>0$
when
\[
\langle F(x)-F(y),x-y\rangle\geq\sigma\|x-y\|^{2}\quad\text{for all }x,y\in\mathbb{R}^{m}.
\]
 The function $F$ is called \emph{Lipschitz continuous} if there
is some constant $L\geq0$ such that
\[
\|F(x)-F(y)\|\leq L\|x-y\|,\quad\text{for all }x,y\in\mathbb{R}^{m}.
\]
 $F$ is called locally Lipschitz continuous if for every $x$ in
$\mathbb{R}^{m}$, there exists a neighbourhood $U$ of $x$ such
that $F$ restricted to $U$ is Lipschitz continuous.

We have the following well-known result.
\begin{proposition}
Let $f:\mathbb{R}^{m}\to\mathbb{R}$ be a differentiable function.
Then f is (strongly) convex if and only if $\nabla f$~is (strongly)
monotone.
\end{proposition}

To establish our convergence results, we will make use of the \emph{\L{}ojasiewicz
property}, defined next.
\begin{definition}
Let $f:\mathbb{R}^{n}\to\mathbb{R}$ be a differentiable function. \end{definition}
\begin{enumerate}
\item The function $f$ is said to have the \emph{\L{}ojasiewicz property}
if for any critical point $\bar{x}$, there exist constants $M>0,\varepsilon>0$
and $\theta\in[0,1)$ such that
\begin{equation}
|f(x)\text{\textminus}f(\bar{x})|^{\theta}\text{\ensuremath{\le}}M\left\Vert \nabla f(x)\right\Vert ,\quad\text{for all }x\in\mathbb{B}(\bar{x},\varepsilon),\label{eq:Lojasiewicz_Inequality}
\end{equation}
where we adopt the convention~$0^{0}=1$. The constant $\theta$
is called \emph{\L{}ojasiewicz exponent} of $f$ at $\bar{x}.$
\item The function $f$ is said to be \emph{real analytic}\textbf{ }if for
every~$x\in\mathbb{R}^{n}$, $f$~may be represented by a convergent
power series in some neighbourhood of $x$. \end{enumerate}
\begin{proposition}
\label{prop:real_analytic}\cite{lojasiewicz1965ensembles} Every
real analytic function $f:\mathbb{R}^{n}\to\mathbb{R}$ satisfies
the \L{}ojasiewicz property with exponent $\theta\in\left[0,1\right)$.
\end{proposition}
Problem~\eqref{eq:DC1} can be easily transformed into an equivalent
problem involving strongly convex functions. Indeed, choose any $\rho>0$
and consider the functions $g(x):=f_{1}(x)+\frac{\rho}{2}\|x\|^{2}$
and $h(x):=f_{2}(x)+\frac{\rho}{2}\|x\|^{2}$. Then $g$ and $h$
are strongly convex functions with modulus $\rho$ and $g(x)-h(x)=\phi(x)$,
for all $x\in\mathbb{R}^{m}$. In this way, we obtain the equivalent
problem
\begin{equation}
\left(\mathcal{P}\right)\ \text{\ensuremath{\underset{x\in\mathbb{R}^{m}}{\textrm{minimise}}}}\;\phi(x)=g(x)-h(x).\label{eq:DC_2norm}
\end{equation}

The key step to solve~$\left(\mathcal{P}\right)$ with DCA is to
approximate the concave part $-h$ of the objective function~$\phi$
by its affine majorisation and then minimise the resulting convex
function. The algorithm proceeds as follows.
%
\begin{center}
\bgroup 	\renewcommand\theenumi{\arabic{enumi}.} 	 \renewcommand\labelenumi{\theenumi}
\fbox{%
\begin{minipage}[t]{.98\textwidth}%
\textbf{ALGORITHM 1:} (DCA, \cite{an_numerical_1996})
\begin{enumerate}
\item Let $x_{0}$ be any initial point and set $k:=0$.
\item Solve the strongly convex optimisation problem
\[
\left(\mathcal{P}_{k}\right)\ \underset{x\in\mathbb{R}^{m}}{\textrm{minimise}}\; g(x)-\langle\nabla h(x_{k}),x\rangle.
\]
 to obtain the unique solution $y_{k}$.
\item If $y_{k}=x_{k}$ then STOP and RETURN $x_{k}$, otherwise set $x_{k+1}:=y_{k}$,
set $k:=k+1$, and go to Step~2.\medskip{}
 \end{enumerate}
\end{minipage}} \egroup{}
\end{center}

In \cite{fukushima_generalized_1981} Fukushima and Mine adapted their
original algorithm reported in \cite{mine_minimization_1981} by adding
a proximal term $\frac{\rho}{2}\left\Vert x-x_{k}\right\Vert ^{2}$
to the objective of the convex optimisation subproblem. As a result
they obtain an optimisation subproblem that is identical to the one
in Step 2 of DCA, when one transforms (\ref{eq:DC1}) into (\ref{eq:DC_2norm})
by adding $\frac{\rho}{2}\|x\|^{2}$ to each convex function. In contrast
to DCA, Fukushima--Mine algorithm~\cite{fukushima_generalized_1981}
also includes a line search along the direction $d_{k}:=y_{k}-x_{k}$
to find the smallest nonnegative integer $l_{k}$ such that the Armijo
type rule
\begin{equation}
\phi(x_{k}+\beta^{l_{k}}d_{k})\leq\phi(x_{k})-\alpha\beta^{l_{k}}\left\Vert d_{k}\right\Vert ^{2}\label{eq:Armijo}
\end{equation}
is satisfied, where $\alpha>0$ and $0<\beta<1$. Thus, when~$l_{k}=0$
satisfies~\eqref{eq:Armijo}, i.e. when
\[
\phi(y_{k})\leq\phi(x_{k})-\alpha\left\Vert d_{k}\right\Vert ^{2},
\]
one has~$x_{k+1}=y_{k}$ and the iterations of both algorithms coincide.
As we shall see in Proposition~\ref{prop:main_inequality}, this
is guaranteed to happen if~$\alpha\leq\rho$.

\section{Boosted DC Algorithms\label{sec:BDCA}}

Let us introduce our first algorithm to solve $(\mathcal{P})$, which
we call a \emph{Boosted DC Algorithm} \emph{with Backtracking}. The
algorithm is a combination of Algorithm~1 and the algorithm of Fukushima--Mine~\cite{fukushima_generalized_1981}.
\begin{center}
\bgroup 	\renewcommand\theenumi{\arabic{enumi}.} 	 \renewcommand\labelenumi{\theenumi}
\fbox{%
\begin{minipage}[t]{.98\textwidth}%
\textbf{ALGORITHM 2:} (BDCA-Backtracking)
\begin{enumerate}
\item Fix $\alpha>0$, $\bar{\lambda}>0$ and $0<\beta<1$. Let $x_{0}$
be any initial point and set $k:=0$.
\item Solve the strongly convex minimisation problem
\[
\left(\mathcal{P}_{k}\right)\underset{x\in\mathbb{R}^{m}}{\textrm{\;\ minimise}}\; g(x)-\langle\nabla h(x_{k}),x\rangle
\]
 to obtain the unique solution $y_{k}$.
\item Set $d_{k}:=y_{k}-x_{k}$. If $d_{k}=0$, STOP and RETURN $x_{k}$.
Otherwise, go to Step~4.
\item Set $\lambda_{k}:=\bar{\lambda}$. WHILE $\phi(y_{k}+\lambda_{k}d_{k})>\phi(y_{k})-\alpha\lambda_{k}\|d_{k}\|^{2}$
DO $\lambda_{k}:=\beta\lambda_{k}$.
\item Set $x_{k+1}:=y_{k}+\lambda_{k}d_{k}$. If $x_{k+1}=x_{k}$ then STOP
and RETURN $x_{k}$, otherwise set $k:=k+1$, and go to Step~2.\medskip{}
 \end{enumerate}
\end{minipage}}
\egroup{}
\end{center}

The next proposition shows that the solution of $(\mathcal{P}_{k})$,
which coincides with the DCA subproblem in Algorithm 1, provides a
decrease in the value of the objective function. For the sake of completeness,
we include its short proof.
\begin{proposition}
\label{prop:main_inequality}For all $k\in\mathbb{N}$, it holds that
\begin{equation}
\phi(y_{k})\leq\phi(x_{k})-\rho\|d_{k}\|^{2}.\label{eq:main_inequality_dca}
\end{equation}
\end{proposition}
\begin{proof}
Since $y_{k}$ is the unique solution of the strongly convex problem
$\left(\mathcal{P}_{k}\right)$, we have
\begin{equation}
\nabla g(y_{k})=\nabla h(x_{k}),\label{eq:optimalcondition}
\end{equation}
 which implies
\[
g(x_{k})-g(y_{k})\geq\langle\nabla h(x_{k}),x_{k}-y_{k}\rangle+\frac{\rho}{2}\|x_{k}-y_{k}\|^{\text{2}}.
\]
On the other hand, the strong convexity of $h$ implies
\[
h(y_{k})-h(x_{k})\geq\langle\nabla h(x_{k}),y_{k}-x_{k}\rangle+\frac{\rho}{2}\|y_{k}-x_{k}\|^{\text{2}}.
\]
Adding the two previous inequalities, we have
\[
g(x_{k})-g(y_{k})+h(y_{k})-h(x_{k})\geq\rho\|x_{k}-y_{k}\|^{2},
\]
which implies \eqref{eq:main_inequality_dca}.
\end{proof}
If $\lambda_{k}=0$, the iterations of BDCA-Backtracking coincide
with those of DCA, since the latter sets $x_{k+1}:=y_{k}$. Next we
show that $d_{k}=y_{k}-x_{k}$ is a descent direction for $\phi$
at $y_{k}$. Thus, one can achieve a larger decrease in the value
of $\phi$ by moving along this direction. This simple fact, which
permits an improvement in the performance of DCA, constitutes the
key idea of our algorithms.
\begin{proposition}
\label{prop:descent_direction}For all $k\in\mathbb{N}$, we have
\begin{equation}
\langle\nabla\phi(y_{k}),d_{k}\rangle\leq-\rho||d_{k}||^{2};\label{eq:descent_direction}
\end{equation}
that is, $d_{k}$ is a descent direction for $\phi$ at $y_{k}$.\end{proposition}
\begin{proof}
The function $h$ is strongly convex with constant $\rho$. This implies
that $\nabla h$ is strongly monotone with constant $\rho$; whence,
\[
\langle\nabla h(x_{k})-\nabla h(y_{k}),x_{k}-y_{k}\rangle\geq\rho\|x_{k}-y_{k}\|^{\text{2}}.
\]
Further, since $y_{k}$ is the unique solution of the strongly convex
problem $\left(\mathcal{P}_{k}\right)$, we have
\[
\nabla h(x_{k})=\nabla g(y_{k}),
\]
 which implies,
\[
\langle\nabla\phi(y_{k}),d_{k}\rangle=\langle\nabla g(y_{k})-\nabla h(y_{k}),d_{k}\rangle\leq-\rho\|d_{k}\|^{\text{2}},
\]
and completes the proof.
\end{proof}
\begin{remark}\label{rem:nonsmooth}
In general, Proposition~\ref{prop:descent_direction} does not remain valid when $g$ is not differentiable. In fact, the direction $d_k$ might be an ascent direction, in which case Step 4 in Algorithm~4 could become an infinite loop.
For instance, consider $g(x)=|x|+\frac{1}{2}x^2+\frac{1}{2}x$ and $h(x)=\frac{1}{2}x^2$ for $x\in\mathbb{R}$. If $x_0=\frac{1}{2}$, one has
\[
\left(\mathcal{P}_{0}\right)\underset{x\in\mathbb{R}}{\textrm{\;\ minimise}}\; |x|+\frac{1}{2}x^2+\frac{1}{2}x-\frac{1}{2}x,
\]
whose unique solution is $y_0=0$. Then, the one-sided directional derivative of $\phi$ at $y_0$ in the direction  $d_0=y_0-x_0=-\frac{1}{2}$ is given by
$$\phi'(y_0;d_0)=\lim_{t\downarrow 0} \frac{\phi(0+t(-1/2))-\phi(0)}{t}=\frac{1}{4}.$$
Thus, $d_0$ is an ascent direction for $\phi$ at $y_0$ (actually, $y_0$ is the global minimum of $\phi$).
\end{remark}
As a corollary, we deduce that the backtracking Step 4 of Algorithm~2
terminates finitely when $\rho>\alpha$.
\begin{corollary}
Suppose that $\rho>\alpha$. Then, for all $k\in\mathbb{N},$ there
is some $\delta_{k}>0$ such that
\begin{equation}
\phi\left(y_{k}+\lambda d_{k}\right)\leq\phi(y_{k})-\alpha\lambda\|d_{k}\|^{2},\quad\text{for all }\lambda\in[0,\delta_{k}].\label{eq:main_inequality_bdca}
\end{equation}
\end{corollary}
\begin{proof}
If $d_{k}=0$ there is nothing to prove. Otherwise, by the mean value
theorem, there is some~${t_\lambda}\in(0,1)$ such that
\begin{align*}
\phi(y_{k}+\lambda d_{k})-\phi(y_{k}) & =\langle\nabla\phi(y_{k}+{t_\lambda}\lambda d_{k}),\lambda d_{k}\rangle\\
 & =\lambda\langle\nabla\phi(y_{k}),d_{k}\rangle+\lambda\langle\nabla\phi(y_{k}+{t}_\lambda\lambda d_{k})-\nabla\phi(y_{k}),d_{k}\rangle\\
 & \leq-\rho\lambda\|d_{k}\|^{2}+\lambda\|\nabla\phi(y_{k}+{t_\lambda}\lambda d_{k})-\nabla\phi(y_{k})\|\|d_{k}\|.
\end{align*}
As~$\nabla\phi$ is continuous at~$y_{k}$, there is some~$\delta>0$
such that
\[
\|\nabla\phi(z)-\nabla\phi(y_{k})\|\leq(\rho-\alpha)\|d_{k}\|\text{ whenever }\|z-y_{k}\|\leq\delta.
\]
 Since $\|y_{k}+{t_\lambda}\lambda d_{k}-y_{k}\|={t_\lambda}\lambda\|d_{k}\|\leq \lambda\|d_{k}\|$,
then for all~$\lambda\in\left(0,\frac{\delta}{\|d_{k}\|}\right)$,
we deduce
\[
\phi(y_{k}+\lambda d_{k})-\phi(y_{k})\leq-\rho\lambda\|d_{k}\|^{2}+(\rho-\alpha)\lambda\|d_{k}\|^{2}=-\alpha\lambda\|d_{k}\|^{2},
\]
 and the proof is complete.
\end{proof}
\begin{remark}
Notice that $y_{k}+\lambda d_{k}=x_{k}+(1+\lambda)d_{k}$. Therefore,
Algorithm~2 uses the same direction as the Fukushima--Mine algorithm
\cite{fukushima_generalized_1981}, where $x_{k+1}=x_{k}+\beta^{l}d_{k}=\beta^{l}y_{k}+\left(1-\beta^{l}\right)x_{k}$
for some $0<\beta<1$ and some nonnegative integer $l$. The iterations
would be the same if $\beta^{l}=\lambda+1$. Nevertheless, as $0<\beta<1$,
the step size $\lambda=\beta^{l}-1$ chosen in the Fukushima--Mine
algorithm \cite{fukushima_generalized_1981} is always less than or
equal to zero, while in Algorithm 2, only step sizes $\lambda\in\,]0,\bar{\lambda}]$
are explored. Moreover, observe that the Armijo type rule \eqref{eq:Armijo},
as used in~\cite{fukushima_generalized_1981}, searches for an $l_{k}$
such that $\phi(x_{k}+\beta^{l_{k}}d_{k})<\phi(x_{k})$, whereas Algorithm
2 searches for a $\lambda_{k}$ such that $\phi(y_{k}+\lambda_{k}d_{k})<\phi(y_{k})$.
We know from (\ref{eq:main_inequality_dca}) and (\ref{eq:main_inequality_bdca})
that
\[
\phi\left(y_{k}+\lambda d_{k}\right)\leq\phi(y_{k})-\alpha\lambda\|d_{k}\|^{2}\leq\phi(x_{k})-(\rho+\alpha\lambda)\|d_{k}\|^{2};
\]
thus, Algorithm 2 results in a larger decrease in the value of $\phi$
at each iteration than DCA, which sets $\lambda:=0$ and $x_{k+1}:=y_{k}$.
Therefore, a faster convergence of Algorithm 2 compared with DCA is
expected, see Figure~\ref{fig:ex1} and Figure~\ref{fig:comparison_DCA_BDCA_2and3}.
\end{remark}
The following convergence results were inspired by \cite{attouch2009convergence},
which in turn were adapted from the original ideas of \L{}ojasiewicz;
see also~\cite[Section~3.2]{Bolte2013}.
\begin{proposition}
\label{prop:innerloop}For any $x_{0}\in\mathbb{R}^{m}$, either Algorithm~2
returns a stationary point of~$\left(\mathcal{P}\right)$ or it generates
an infinite sequence such that the following holds.
\begin{enumerate}
\item $\phi(x_{k})$ is monotonically decreasing and convergent to some
$\phi^{*}$.
\item Any limit point of $\{x_{k}\}$ is a stationary point of $\left(\mathcal{P}\right)$.
If in addition, $\phi$ is coercive then there exits a subsequence
of $\{x_{k}\}$ which converges to a stationary point of~$\left(\mathcal{P}\right)$.
\item $\sum_{k=0}^{\infty}\|d_{k}\|^{2}<\infty$ and $\sum_{k=0}^{\infty}\|x_{k+1}-x_{k}\|^{2}<\infty$.\end{enumerate}
\end{proposition}
\begin{proof}
Because of~\eqref{eq:optimalcondition}, if Algorithm~2 stops at
Step~3 and returns $x_{k}$, then $x_{k}$ must be a stationary point
of~$\left(\mathcal{P}\right)$. Otherwise, by Proposition~\ref{prop:main_inequality}
and Step~4 of Algorithm~2, we have
\begin{equation}
\phi(x_{k+1})\leq\phi(y_{k})-\alpha\lambda_{k}\|d_{k}\|^{2}\leq\phi(x_{k})-(\alpha\lambda_{k}+\rho)\|d_{k}\|^{2}.\label{eq:phi_decreasing}
\end{equation}
Hence, as the sequence $\{\phi(x_{k})\}$ is monotonically decreasing
and bounded from below by~\eqref{eq:phi_bounded_below}, it converges
to some $\phi^{*}$, which proves~(i). Consequently, we have
\[
\phi(x_{k+1})-\phi(x_{k})\to0.
\]
Thus, by~\eqref{eq:phi_decreasing}, one has $\|d_{k}\|^{2}=\|y_{k}-x_{k}\|^{2}\to0.$

Let $\bar{x}$ be any limit point of $\{x_{k}\}$, and let $\{x_{k_{i}}\}$
be a subsequence of $\{x_{k}\}$ converging to $\bar{x}$. Since $\|y_{k_{i}}-x_{k_{i}}\|\to0$,
one has
\[
y_{k_{i}}\to\bar{x}.
\]
Taking the limit as $i\to\infty$ in~\eqref{eq:optimalcondition},
as $\nabla h$ and $\nabla g$ are continuous, we have $\nabla h(\bar{x})=\nabla g(\bar{x})$.

If $\phi$ is coercive, since the sequence $\{\phi(x_{k})\}$ is convergent,
then the sequence $\{x_{k}\}$ is bounded. This implies that there
exits a subsequence of $\{x_{k}\}$ converging to $\bar{x}$, a stationary
point of $\left(\mathcal{P}\right)$, which proves~(ii).

To prove~(iii), observe that~\eqref{eq:phi_decreasing} implies
that
\begin{equation}
(\alpha\lambda_{k}+\rho)\|d_{k}\|^{2}\leq\phi(x_{k})-\phi(x_{k+1}).\label{eq:ineq_descent_phi}
\end{equation}
Summing this inequality from~$0$ to~$N$, we obtain
\begin{equation}
\sum_{k=0}^{N}(\alpha\lambda_{k}+\rho)\|d_{k}\|^{2}\leq\phi(x_{0})-\phi(x_{N+1})\leq\phi(x_{0})-\inf_{x\in\mathbb{R}^{m}}\phi(x),\label{eq:bound_sum}
\end{equation}
whence, taking the limit when $N\to\infty,$
\[
\sum_{k=0}^{\infty}\rho\|d_{k}\|^{2}\leq\sum_{k=0}^{\infty}(\alpha\lambda_{k}+\rho)\|d_{k}\|^{2}\leq\phi(x_{0})-\inf_{x\in\mathbb{R}^{m}}\phi(x)<\infty,
\]
 so we have $\sum_{k=0}^{\infty}\|d_{k}\|^{2}<\infty$. Since
\[
x_{k+1}-x_{k}=y_{k}-x_{k}+\lambda_{k}d_{k}=(1+\lambda_{k})d_{k},
\]
we obtain
\[
\sum_{k=0}^{\infty}\|x_{k+1}-x_{k}\|^{2}=\sum_{k=0}^{\infty}(1+\lambda_{k})^{2}\|d_{k}\|^{2}\leq(1+\bar{\lambda})^{2}\sum_{k=0}^{\infty}\|d_{k}\|^{2}<\infty,
\]
and the proof is complete.
\end{proof}
We will employ the following useful lemma to obtain bounds on the
rate of convergence of the sequences generated by Algorithm~2. This
result appears within the proof of~\cite[Theorem~2]{attouch2009convergence}
for specific values of~$\alpha$ and~$\beta$. See also~\cite[Theorem~3.3]{le2009convergence},
or very recently, \cite[Theorem~3]{Li2014}.
\begin{lemma}
\label{lem:rate_convergence}Let~$\left\{ s_{k}\right\} $ be a sequence
in~$\mathbb{R}_{+}$ and let~$\alpha,\beta$ be some positive constants.
Suppose that $s_{k}\to0$ and that the sequence satisfies
\begin{equation}
s_{k}^{\alpha}\leq\beta(s_{k}-s_{k+1}),\quad\text{for all }k\text{ sufficiently large.}\label{eq:ineq_seq}
\end{equation}
Then
\begin{enumerate}
\item if~$\alpha=0$, the sequence~$\left\{ s_{k}\right\} $ converges
to~$0$ in a finite number of steps;
\item if~$\alpha\in\left(0,1\right]$, the sequence~$\left\{ s_{k}\right\} $
converges linearly to~$0$ with rate~$1-\frac{1}{\beta}$;
\item if~$\alpha>1$, there exists~$\eta>0$ such that
\[
s_{k}\leq\eta k^{-\frac{1}{\alpha-1}},\quad\text{for all }k\text{ sufficiently large.}
\]
\end{enumerate}
\end{lemma}
\begin{proof}
If~$\alpha=0$, then~\eqref{eq:ineq_seq} implies
\[
0\leq s_{k+1}\leq s_{k}-\frac{1}{\beta},
\]
 and~(i) follows.

Assume that~$\alpha\in(0,1]$. Since~$s_{k}\to0$, we have that~$s_{k}<1$
for all~$k$ large enough. Thus, by~\eqref{eq:ineq_seq}, we have
\[
s_{k}\leq s_{k}^{\alpha}\leq\beta(s_{k}-s_{k+1}).
\]
Therefore, $s_{k+1}\leq\left(1-\frac{1}{\beta}\right)s_{k}$; i.e.,
$\left\{ s_{k}\right\} $ converges linearly to~$0$ with rate~$1-\frac{1}{\beta}$.

Suppose now that~$\alpha>1$. If~$s_{k}=0$ for some~$k$, then~\eqref{eq:ineq_seq}
implies~$s_{k+1}=0$. Then the sequence converges to zero in a finite
number of steps, and thus~(iii) trivially holds. Hence, we will assume
that~$s_{k}>0$ and that~\eqref{eq:ineq_seq} holds for all~$k\geq N$,
for some positive integer~$N$. Consider the decreasing function
$\varphi:(0,+\infty)\to\mathbb{R}$ defined by $\varphi(s):=s^{-\alpha}$.
By~\eqref{eq:ineq_seq}, for $k\geq N$, we have
\[
\frac{1}{\beta}\leq\left(s_{k}-s_{k+1}\right)\varphi(s_{k})\leq\int_{s_{k+1}}^{s_{k}}\varphi(t)dt=\frac{s_{k+1}^{1-\alpha}-s_{k}^{1-\alpha}}{\alpha-1}.
\]
As~$\alpha-1>0$, this implies that
\[
s_{k+1}^{1-\alpha}-s_{k}^{1-\alpha}\geq\frac{\alpha-1}{\beta},
\]
for all $k\geq N$. Thus, summing for $k$ from $N$ to $j-1\geq N$,
we have
\[
s_{j}^{1-\alpha}-s_{N}^{1-\alpha}\geq\frac{\alpha-1}{\beta}(j-N),
\]
which gives, for all $j\geq N+1$,
\[
s_{j}\leq\left(s_{N}^{1-\alpha}+\frac{\alpha-1}{\beta}(j-N)\right)^{\frac{1}{1-\alpha}}.
\]
Therefore, there is some~$\eta>0$ such that
\[
s_{j}\leq\eta j^{-\frac{1}{\alpha-1}},\quad\text{for all }k\text{ sufficiently large,}
\]
which completes the proof.
\end{proof}
\begin{theorem}
\label{th:convergence_Lowasiewicz} Suppose that $\nabla g$ is locally
Lipschitz continuous and $\phi$ satisfies the \L{}ojasiewicz property
with exponent $\theta\in\left[0,1\right)$. For any $x_{0}\in\mathbb{R}^{m}$,
consider the sequence $\left\{ x_{k}\right\} $ generated by Algorithm~2.
If the sequence $\{x_{k}\}$ has a cluster point~$x^{*}$, then the
whole sequence converges to~$x^{*},$ which is a stationary point
of~$\left(\mathcal{P}\right)$. Moreover, denoting~$\phi^{*}:=\phi(x^{*})$,
the following estimations hold:
\begin{enumerate}
\item if $\theta=0$ then the sequences~$\{x_{k}\}$ and $\{\phi(x_{k})\}$
converge in a finite number of steps to~$x^{*}$ and $\phi^{*}$,
respectively;
\item if $\theta\in\left(0,\frac{1}{2}\right]$ then the sequences~$\{x_{k}\}$
and $\{\phi(x_{k})\}$ converge linearly to~$x^{*}$ and $\phi^{*}$,
respectively;
\item if $\theta\in\left(\frac{1}{2},1\right)$ then there exist some positive
constants $\eta_{1}$ and $\eta_{2}$ such that
\begin{gather*}
\|x_{k}-x^{*}\|\leq\eta_{1}k^{-\frac{1-\theta}{2\theta-1}},\\
\phi(x_{k})-\phi^{*}\leq\eta_{2}k^{-\frac{1}{2\theta-1}},
\end{gather*}
for all large~$k$. \end{enumerate}
\end{theorem}
\begin{proof}
By Proposition \ref{prop:innerloop}, we have $\lim_{k\to\infty}\phi(x_{k})=\phi^{*}$
. If $x^{*}$ is a cluster point of~$\left\{ x_{k}\right\} $, then
there exists a subsequence $\{x_{k_{i}}\}$ of $\{x_{k}\}$ that converges
to $x^{*}$. By continuity of $\phi$, we have that
\[
\phi(x^{*})=\lim_{i\to\infty}\phi(x_{k_{i}})=\lim_{k\to\infty}\phi(x_{k})=\phi^{*}.
\]
Hence, $\phi$ is finite and has the same value~$\phi^{*}$ at every
cluster point of $\{x_{k}\}$. If $\phi(x_{k})=\phi^{*}$ for some
$k>1$, then $\phi(x_{k})=\phi(x_{k+p})$ for any $p\geq0$, since
the sequence $\phi(x_{k})$ is decreasing. Therefore, $x_{k}=x_{k+p}$
for all $p\geq0$ and Algorithm 2 terminates after a finite number
of steps. From now on, we assume that $\phi(x_{k})>\phi^{*}$ for
all $k$.

As~$\phi$ satisfies the \L{}ojasiewicz property, there exist $M>0,\varepsilon_{1}>0$
and $\theta\in[0,1)$ such that
\begin{equation}
|\phi(x)\text{\textminus}\phi(x^{*})|^{\theta}\text{\ensuremath{\le}}M\left\Vert \nabla\phi(x)\right\Vert ,\quad\forall x\in\mathbb{B}(x^{*},\varepsilon_{1}).\label{LojasiewiczProperty}
\end{equation}
Further, as $\nabla g$ is locally Lipschitz around $x^{*}$, there
are some constants $L\geq0$ and $\varepsilon_{2}>0$ such that
\begin{equation}
\|\nabla g(x)-\nabla g(y)\|\leq L\|x-y\|,\quad\forall x,y\in\mathbb{B}(x^{*},\varepsilon_{2}).\label{eq:nabla_g_Lip}
\end{equation}
Let $\varepsilon:=\frac{1}{2}\min\left\{ \varepsilon_{1},\varepsilon_{2}\right\} >0$.
Since $\lim_{i\to\infty}x_{k_{i}}=x^{*}$ and $\lim_{i\to\infty}\phi(x_{k_{i}})=\phi^{*}$,
we can find an index $N$ large enough such that
\begin{equation}
\|x_{N}-x^{*}\|+\frac{ML\left(1+\bar{\lambda}\right)}{(1-\theta)\rho}\left(\phi(x_{N})-\phi^{*}\right)^{1-\theta}<\varepsilon.\label{eq:BallCondition}
\end{equation}
By Proposition~\ref{prop:innerloop}(iii), we know that $d_{k}=y_{k}-x_{k}\to0$.
Then, taking a larger $N$ if needed, we can assure that
\[
\|y_{k}-x_{k}\|\leq\varepsilon,\quad\forall k\geq N.
\]
We now prove that, for all $k\geq N$, whenever $x_{k}\in\mathbb{B}(x^{*},\varepsilon)$
it holds
\begin{align}
\|x_{k+1}-x_{k}\| & \leq\frac{ML\left(1+\lambda_{k}\right)}{(1-\theta)(\alpha\lambda_{k}+\rho)}\left[\left(\phi(x_{k})-\phi^{*}\right)^{1-\theta}-\left(\phi(x_{k+1})-\phi^{*}\right)^{1-\theta}\right]\label{eq:StepRelation}
\end{align}
 Indeed, consider the concave function $\gamma:(0,+\infty)\to(0,+\infty)$
defined as $\gamma(t):=t^{1-\theta}$. Then, we have
\[
\gamma(t_{1})-\gamma(t_{2})\geq\nabla\gamma(t_{1})^{T}(t_{1}-t_{2}),\quad\forall t_{1},t_{2>0.}
\]
 Substituting in this inequality $t_{1}$ by $\left(\phi(x_{k})-\phi^{*}\right)$
and $t_{2}$ by $\left(\phi(x_{k+1})-\phi^{*}\right)$ and using~\eqref{LojasiewiczProperty}
and then~\eqref{eq:ineq_descent_phi}, one has
\begin{align}
\left(\phi(x_{k})-\phi^{*}\right)^{1-\theta}-\left(\phi(x_{k+1})-\phi^{*}\right)^{1-\theta} & \geq\frac{1-\theta}{\left(\phi(x_{k})-\phi^{*}\right)^{\theta}}\left(\phi(x_{k})-\phi(x_{k+1})\right)\nonumber \\
 & \geq\frac{1-\theta}{M\left\Vert \nabla\phi(x_{k})\right\Vert }\left(\alpha\lambda_{k}+\rho\right)\|y_{k}-x_{k}\|^{2}\nonumber \\
 & =\frac{\left(1-\theta\right)\left(\alpha\lambda_{k}+\rho\right)}{M\left(1+\lambda_{k}\right)^{2}\left\Vert \nabla\phi(x_{k})\right\Vert }\|x_{k+1}-x_{k}\|^{2}.\label{eq:LojasiewiczInequality}
\end{align}
 On the other hand, since $\nabla g(y_{k})=\nabla h(x_{k})$ and
\[
\|y_{k}-x^{*}\|\leq\|y_{k}-x_{k}\|+\|x_{k}-x^{*}\|\leq2\varepsilon\leq\varepsilon_{2},
\]
using~\eqref{eq:nabla_g_Lip}, we obtain
\begin{align}
\left\Vert \nabla\phi(x_{k})\right\Vert  & =\left\Vert \nabla g(x_{k})-\nabla h(x_{k})\right\Vert =\left\Vert \nabla g(x_{k})-\nabla g(y_{k})\right\Vert \nonumber \\
 & \leq L\left\Vert x_{k}-y_{k}\right\Vert =\frac{L}{(1+\lambda_{k})}\left\Vert x_{k+1}-x_{k}\right\Vert .\label{eq:ineq2}
\end{align}
 Combining~\eqref{eq:LojasiewiczInequality} and~\eqref{eq:ineq2},
we obtain (\ref{eq:StepRelation}).

From~(\ref{eq:StepRelation}), as $\lambda_{k}\in(0,\bar{\lambda]}$,
we deduce
\begin{equation}
\|x_{k+1}-x_{k}\|\leq\frac{ML\left(1+\bar{\lambda}\right)}{(1-\theta)\rho}\left[\left(\phi(x_{k})-\phi^{*}\right)^{1-\theta}-\left(\phi(x_{k+1})-\phi^{*}\right)^{1-\theta}\right],\label{eq:StepRelation2}
\end{equation}
for all $k\geq N$ such that $x_{k}\in\mathbb{B}(x^{*},\varepsilon).$

We prove by induction that $x_{k}\in\mathbb{B}(x^{*},\varepsilon)$
for all $k\geq N$. Indeed, from (\ref{eq:BallCondition}) the claim
holds for $k=N$. We suppose that it also holds for $k=N,N+1,\ldots,N+p-1$,
with~$p\geq1$. Then~(\ref{eq:StepRelation2}) is valid for $k=N,N+1,\ldots,N+p-1$.
Therefore
\begin{align*}
\left\Vert x_{N+p}-x^{*}\right\Vert  & \leq\left\Vert x_{N}-x^{*}\right\Vert +\sum_{i=1}^{p}\left\Vert x_{N+i}-x_{N+i-1}\right\Vert \\
 & \leq\left\Vert x_{N}-x^{*}\right\Vert +\frac{ML\left(1+\bar{\lambda}\right)}{(1-\theta)\rho}\sum_{i=1}^{p}\left[\left(\phi(x_{N+i-1})-\phi^{*}\right)^{1-\theta}-\left(\phi(x_{N+i})-\phi^{*}\right)^{1-\theta}\right]\\
 & \leq\left\Vert x_{N}-x^{*}\right\Vert +\frac{ML\left(1+\bar{\lambda}\right)}{(1-\theta)\rho}\left(\phi(x_{N})-\phi^{*}\right)^{1-\theta}<\varepsilon,
\end{align*}
 where the last inequality follows from (\ref{eq:BallCondition}).

Adding (\ref{eq:StepRelation2}) from $k=N$ to $P$ one has
\begin{equation}
\sum_{k=N}^{P}\|x_{k+1}-x_{k}\|\leq\frac{ML\left(1+\bar{\lambda}\right)}{(1-\theta)\rho}\left(\phi(x_{N})-\phi^{*}\right)^{1-\theta}.\label{eq:bound_finite_sum}
\end{equation}
Taking the limit as $P\to\infty$, we can conclude that
\begin{equation}
\sum_{k=1}^{\infty}\|x_{k+1}-x_{k}\|<\infty.\label{eq:finite_sum}
\end{equation}
This means that $\{x_{k}\}$ is a Cauchy sequence. Therefore, since~$x^{*}$
is a cluster point of~$\{x_{k}\}$, the whole sequence $\{x_{k}\}$
converges to~$x^{*}$. By Proposition \ref{prop:innerloop}, $x^{*}$
must be a stationary point of~$\left(\mathcal{P}\right)$.

For $k\geq N$, it follows from \eqref{LojasiewiczProperty},~\eqref{eq:nabla_g_Lip}
and~\eqref{eq:ineq_descent_phi} that
\begin{align}
(\phi(x_{k})-\phi^{*})^{2\theta} & \leq M^{2}\left\Vert \nabla\phi(x_{k})\right\Vert ^{2}\nonumber \\
 & \leq M^{2}\left\Vert \nabla g(x_{k})-\nabla h(x_{k})\right\Vert ^{2}=M^{2}\left\Vert \nabla g(x_{k})-\nabla g(y_{k})\right\Vert ^{2}\nonumber \\
 & \leq M^{2}L^{2}\left\Vert x_{k}-y_{k}\right\Vert ^{2}\leq\frac{M^{2}L^{2}}{\alpha\lambda_{k}+\rho}\left[\phi(x_{k})-\phi(x_{k+1})\right]\nonumber \\
 & \leq\delta\left[\left(\phi(x_{k})-\phi^{*}\right)-\left(\phi(x_{k+1})-\phi^{*}\right)\right],\label{eq:main_ineq_conv_rate}
\end{align}
where $\delta:=\frac{M^{2}L^{2}}{\rho}>0$. By applying Lemma~\ref{lem:rate_convergence}
with~$s_{k}:=\phi(x_{k})-\phi^{*}$, $\alpha:=2\theta$ and~$\beta:=\delta$,
statements~(i)-(iii) regarding the sequence~$\left\{ \phi(x_{k})\right\} $
easily follow from~\eqref{eq:main_ineq_conv_rate}.

We know that $s_{i}:=\sum_{k=i}^{\infty}\|x_{k+1}-x_{k}\|$ is finite
by~\eqref{eq:finite_sum}. Notice that~$\|x_{i}-x^{*}\|\leq s_{i}$
by the triangle inequality. Therefore, the rate of convergence of~$x_{i}$
to $x^{*}$ can be deduced from the convergence rate of $s_{i}$ to~0.
Adding~\eqref{eq:StepRelation2} from~$i$ to~$P$ with $N\leq i\leq P$,
we have
\[
s_{i}=\lim_{P\to\infty}\sum_{k=i}^{P}\|x_{k+1}-x_{k}\|\leq K_{1}\left(\phi(x_{i})-\phi^{*}\right)^{1-\theta},
\]
where~$K_{1}:=\frac{ML\left(1+\bar{\lambda}\right)}{\left(1-\theta\right)\rho}>0$.
Then by~\eqref{LojasiewiczProperty} and~\eqref{eq:nabla_g_Lip},
we get
\begin{align*}
s_{i}^{\frac{\theta}{1-\theta}} & \leq MK_{1}^{\frac{\theta}{1-\theta}}\|\nabla\phi(x_{i})\|\leq MLK_{1}^{\frac{\theta}{1-\theta}}\|x_{i}-y_{i}\|\\
 & \leq\frac{MLK_{1}^{\frac{\theta}{1-\theta}}}{1+\lambda_{i}}\|x_{i+1}-x_{i}\|\leq MLK_{1}^{\frac{\theta}{1-\theta}}\|x_{i+1}-x_{i}\|\\
 & =MLK_{1}^{\frac{\theta}{1-\theta}}\left(s_{i}-s_{i+1}\right)
\end{align*}
Hence, taking $K_{2}:=MLK_{1}^{\frac{\theta}{1-\theta}}>0$, for all~$i\geq N$
we have
\[
s_{i}^{\frac{\theta}{1-\theta}}\leq K_{2}\left(s_{i}-s_{i+1}\right).
\]
By applying Lemma~\ref{lem:rate_convergence} with~$\alpha:=\frac{\theta}{1-\theta}$
and~$\beta:=K_{2}$, we see that the statements in~(i)-(iii) regarding
the sequence~$\left\{ x_{k}\right\} $ hold.
\end{proof}

\begin{example}
\label{ex:1}Consider the function $\phi(x)=\frac{1}{4}x^{4}-\frac{1}{2}x^{2}$.
The iteration given by DCA (Algorithm~1) satisfies
\[
x_{k+1}^{3}-x_{k}=0;
\]
that is, $x_{k+1}=\sqrt[3]{x_{k}}$. On the other hand, the iteration
defined by Algorithm~2 is
\[
\widetilde{x_{k+1}}=(1+\lambda_{k})\sqrt[3]{\widetilde{x_{k}}}-\lambda_{k}\widetilde{x_{k}}.
\]
If $x_{0}=\widetilde{x_{0}}=\frac{27}{125}$, we have $x_{1}=\frac{3}{5}$,
while $\widetilde{x_{1}}=\frac{3}{5}(1+\lambda_{0})-\frac{27}{125}\lambda_{0}$.
For any $\lambda_{0}\in\left(0,\frac{25\sqrt{41}-75}{48}\right]$,
we have $\phi(\widetilde{x_{1}})<\phi(x_{1})$. The optimal step size
is attained at $\lambda_{{\rm opt}}=\frac{25}{24}$ with $x_{1}=1$,
which is the global minimiser of~$\phi$.

\begin{figure}[H]
\begin{centering}
\includegraphics[width=0.7\textwidth]{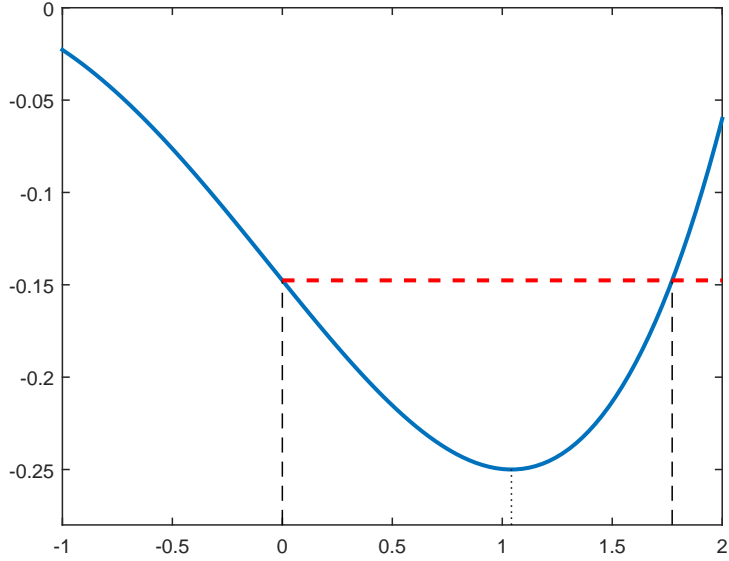}
\par\end{centering}

\protect\caption{\label{fig:ex1}Plot of $\phi\left(\frac{3}{5}(1+\lambda)-\frac{27}{125}\lambda\right)$
for $\phi(x)=\frac{1}{4}x^{4}-\frac{1}{2}x^{2}$. As shown in Proposition~\ref{prop:descent_direction},
the function is decreasing at~0. The value $\lambda=0$ corresponds
to the next iteration chosen by DCA, while the next iteration chosen
by algorithm~\cite{fukushima_generalized_1981} sets $\lambda\in\,]-1,0]$.
Algorithm~2 chooses the next iteration taking~$\lambda\in\,]0,\bar{\lambda}]$,
which permits to achieve an additional decrease in the value of~$\phi$.
Here, the optimal value is attained at~$\lambda_{{\rm opt}}=\frac{25}{24}\approx1.04.$}
\end{figure}
Observe in Figure~\ref{fig:ex1} that the function
\[
\phi_{k}(\lambda):=\phi(y_{k}+\lambda d_{k})
\]
behaves as a quadratic function nearby 0. Then, a quadratic interpolation
of this function should give us a good candidate for choosing a step
size close to the optimal one. Whenever $\nabla\phi$ is not \emph{too
expensive} to compute, it makes sense to construct a quadratic approximation
of $\phi$ with an interpolation using three pieces of information:
$\phi_{k}(0)=\phi(y_{k})$, $\phi_{k}'(0)=\nabla\phi(y_{k})^{T}d_{k}$
and $\phi_{k}(\bar{\lambda})$. This gives us the quadratic function
\begin{equation}
\varphi_{k}(\lambda):=\left(\frac{\phi_{k}(\bar{\lambda})-\phi_{k}(0)-\bar{\lambda}\phi_{k}'(0)}{\bar{\lambda}^{2}}\right)\lambda^{2}+\phi_{k}'(0)\lambda+\phi_{k}(0),\label{eq:quadratic_interp}
\end{equation}
see e.g.~\cite[Section~3.5]{nocedal_numerical_2006}. When $\phi_{k}(\bar{\lambda})>\phi_{k}(0)+\bar{\lambda}\phi_{k}'(0)$,
the function $\varphi_{k}$ has a global minimiser at
\begin{equation}
\widehat{\lambda_{k}}:=-\frac{\phi_{k}'(0)\bar{\lambda}^{2}}{2\left(\phi_{k}(\bar{\lambda})-\phi_{k}(0)-\phi_{k}'(0)\bar{\lambda}\right)}.\label{eq:min_interp}
\end{equation}

This suggests the following modification of Algorithm~2.
\end{example}
\begin{center}
\bgroup 	\renewcommand\theenumi{\arabic{enumi}.} 	 \renewcommand\labelenumi{\theenumi}
\fbox{%
\begin{minipage}[t]{.98\textwidth}%
\textbf{ALGORITHM 3:} (BDCA-Quadratic Interpolation with Backtracking)
\begin{enumerate}
\item Fix $\alpha>0$, $\lambda_{\max}>\bar{\lambda}>0$ and $0<\beta<1$.
Let $x_{0}$ be any initial point and set $k:=0$.
\item Solve the strongly convex minimisation problem
\[
\left(\mathcal{P}_{k}\right)\underset{x\in\mathbb{R}^{m}}{\textrm{\;\ minimise}}\; g(x)-\langle\nabla h(x_{k}),x\rangle
\]
 to obtain the unique solution $y_{k}$.
\item Set $d_{k}:=y_{k}-x_{k}$. If $d_{k}=0$ STOP and RETURN $x_{k}$.
Otherwise, go to Step~4.
\item Compute $\widehat{\lambda_{k}}$ as in~\eqref{eq:min_interp}. If~$\widehat{\lambda_{k}}>0$
and $\phi(y_{k}+\widehat{\lambda_{k}}d_{k})<\phi(y_{k}+\bar{\lambda}d_{k})$
set $\lambda_{k}:=\min\left\{ \widehat{\lambda_{k}},\lambda_{\max}\right\} $;
otherwise, set $\lambda_{k}:=\bar{\lambda}$. \\
WHILE $\phi(y_{k}+\lambda_{k}d_{k})>\phi(y_{k})-\alpha\lambda_{k}\|d_{k}\|^{2}$
DO $\lambda_{k}:=\beta\lambda_{k}$.
\item Set $x_{k+1}:=y_{k}+\lambda_{k}d_{k}$. If $x_{k+1}=x_{k}$ then STOP
and RETURN $x_{k}$, otherwise set $k:=k+1$, and go to Step~2.\medskip{}
 \end{enumerate}
\end{minipage}}
\egroup{}
\end{center}
\begin{corollary}
\label{cor:convergence_Loj}The statements in Theorem~\ref{th:convergence_Lowasiewicz}
also apply to Algorithm~3.\end{corollary}
\begin{proof}
Just observe that the proof of Theorem~\ref{th:convergence_Lowasiewicz}
remains valid as long as the step sizes are bounded above by some
constant and below by zero. Algorithm~3 uses the same directions
than Algorithm~2, and the step sizes chosen by Algorithm~3 are bounded
above by~$\lambda_{\max}$ and below by zero.
\end{proof}
Another option here would be to construct a quadratic approximation
$\psi_{k}$ using $\phi_{k}(-1)=\phi(x_{k})$ instead of $\phi_{k}(\bar{\lambda})$.
This interpolation is computationally less expensive, as it does not
require the computation of $\phi_{k}(\bar{\lambda})$. Nevertheless,
our numerical tests for the functions in Section~\ref{sec:A-DC-problem}
show that this approximation usually fits the function $\phi_{k}$
more poorly. In particular, this situation occurs in Example~\ref{ex:1},
as shown in Figure~\ref{fig:quadratic_interp}.
\begin{figure}[H]
\centering{}\includegraphics[width=0.64\textwidth]{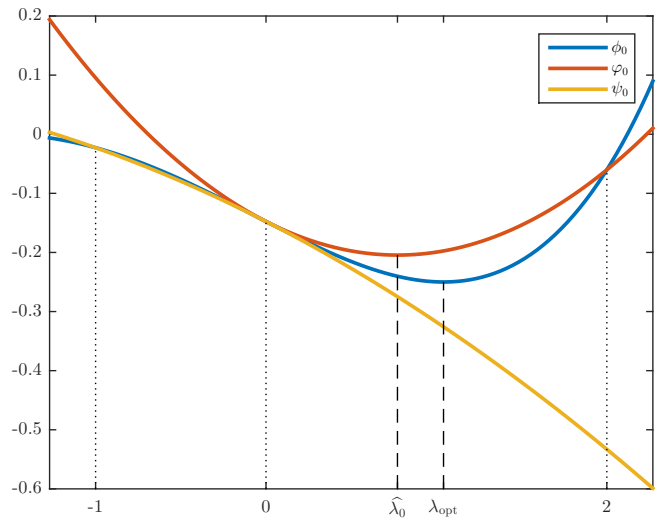}\protect\caption{\label{fig:quadratic_interp}Plots of the quadratic interpolations
$\varphi_{0}$ and $\psi_{0}$ for the function $\phi_{0}(\lambda)=\phi(y_{0}+\lambda d_{0})$
from Example~\ref{ex:1}, with~$\bar{\lambda}=2$. Note that $\psi_{0}(\lambda)$
fits poorly $\phi_{0}(\lambda)$ for $\lambda>0$.}
\end{figure}

One could also construct a cubic function that interpolates $\phi_{k}(-1)$,
$\phi_{k}(0)$, $\phi_{k}'(0)$ and $\phi_{k}(\bar{\lambda})$, see~\cite[Section~3.5]{nocedal_numerical_2006}.
However, for the functions in Section~\ref{sec:A-DC-problem}, we
have observed that this cubic function usually fits the function $\phi_{k}$
worse than the quadratic function $\varphi_{k}$ in~\eqref{eq:quadratic_interp}.
\begin{remark}
Observe that Algorithm 2 and Algorithm 3 still work well if we replace
Step~2 by the following proximal step as in~\cite{moudafi_proximalDC_2006}
\[
\left(\mathcal{P}_{k}\right)\underset{x\in\mathbb{R}^{m}}{\textrm{\;\ minimise}}\; g(x)-\langle\nabla h(x_{k}),x\rangle+\frac{1}{2c_{k}}\|x-x_{k}\|^{\text{2}},
\]
for some positive constants~$c_{k}$.\end{remark}
\begin{example}
(Finding zeroes of systems of DC functions)\label{ex:Finding-zeroes}

Suppose that one wants for find a zero of a system of equations
\begin{equation}
p(x)=c(x),\quad x\in\mathbb{R}^{m}\label{eq:system_eq}
\end{equation}
where $p:\mathbb{R}^{m}\to\mathbb{R}_{+}^{m}$ and $c:\mathbb{R}^{m}\to\mathbb{R}_{+}^{m}$
are twice continuously differentiable functions such that $p_{i}:\mathbb{R}^{m}\to\mathbb{R}_{+}$
and $c_{i}:\mathbb{R}^{m}\to\mathbb{R}_{+}$ are convex functions
for all $i=1,\ldots,m$. Then,
\[
\|p(x)-c(x)\|^{\text{2}}=2\left(\|p(x)\|^{\text{2}}+\|c(x)\|^{\text{2}}\right)-\|p(x)+c(x)\|^{\text{2}}.
\]
Observe that all the components of $p(x)$ and $c(x)$ are nonnegative
convex functions. Hence, both $f_{1}(x):=2\left(\|p(x)\|^{\text{2}}+\|c(x)\|^{\text{2}}\right)$
and $f_{2}(x):=\|p(x)+c(x)\|^{\text{2}}$ are continuously differentiable
convex functions, because they can be expressed as a finite combination
of sums and products of nonnegative convex functions. Thus, we can
either apply DCA or BDCA in order to find a solution to~\eqref{eq:system_eq}
by setting $\phi(x):=f_{1}(x)-f_{2}(x).$

Let~$f(x):=p(x)-c(x)$ for $x\in\mathbb{R}^{m}$. Suppose that $\bar{x}$
is an accumulation point of the sequence $\left\{ x_{k}\right\} $
generated by either Algorithm~2 or Algorithm~3, and assume that~\emph{$\nabla f(\bar{x})$
}is nonsingular. Then, by Proposition~\ref{prop:innerloop}, we must
have $\nabla f(\bar{x})f(\bar{x})=0_{m}$, which implies that~$f(\bar{x})=0_{m}$,
as\emph{~$\nabla f(\bar{x})$ }is nonsingular. Moreover, for all~$x$ close to~$\bar{x}$, we have
\begin{align*}
|\phi(x)-\phi(\bar{x})|^{\frac{1}{2}} & =\phi(x)^{\frac{1}{2}}=\|f(x)\|\\
&=\left\Vert \left(\nabla f(x)\right)^{-1}\nabla f(x)f(x)\right\Vert\\
 &\leq\left\Vert \left(\nabla f(x)\right)^{-1}\right\Vert \left\Vert \nabla f(x)f(x)\right\Vert\\ &=\frac{1}{2}\left\Vert \left(\nabla f(x)\right)^{-1}\right\Vert \|\nabla\phi(x)\|\\
 & \leq M\|\nabla\phi(x)\|,
\end{align*}
 where~$\|\cdot\|$ also denote
the induced matrix norm and $M$ is an upper bound of $\frac{1}{2}\|\left(\nabla f(x)\right)^{-1}\|$
around~$\bar{x}$. Thus,~$\phi$ has the \L{}ojasiewicz property
at~$\bar{x}$ with exponent~$\theta=\frac{1}{2}$. Finally, for
all $\rho>0$, the function $g(x):=f_{1}(x)+\frac{\rho}{2}\|x\|^{2}$
is twice continuously differentiable, which in particular implies
that $\nabla g$ is locally Lipschitz continuous. Therefore, either
Theorem~\ref{th:convergence_Lowasiewicz} or Corollary~\ref{cor:convergence_Loj}
guarantee the linear convergence of~$\left\{ x_{k}\right\} $ to~$\bar{x}$.
\end{example}
\newpage

\section{\noindent A DC problem in biochemistry\label{sec:A-DC-problem}}

Consider a biochemical network with $m$ molecular species and $n$
reversible elementary reactions%
\footnote{An elementary reaction is a chemical reaction for which no intermediate
molecular species need to be postulated in order to describe the chemical
reaction on a molecular scale.%
}. Define forward and reverse \emph{stoichiometric matrices}, $F,R\in\mathbb{\mathbb{Z}}_{\ge0}^{m\times n}$,
respectively, where $F_{ij}$ denotes the \emph{stoichiometry}%
\footnote{Reaction stoichiometry is a quantitative relationship between the
relative quantities of molecular species involved in a single chemical
reaction. %
} of the $i^{th}$ molecular species in the $j^{th}$ forward reaction
and $R_{ij}$ denotes the stoichiometry of the $i^{th}$ molecular
species in the $j^{th}$ reverse reaction. We use the standard inner
product in $\mathbb{R}^{m}$, i.e., $\langle x,y\rangle=x^{T}y$ for
all $x,y\in\mathbb{R}^{m}.$ We assume that \emph{every reaction conserves
mass}, that is, there exists at least one positive vector $l\in\mathbb{R}_{>0}^{m}$
satisfying $(R-F)^{T}l=0_{n}$ \cite{gevorgyan2008detection} where
$R-F$ represents net reaction stoichiometry. We assume the cardinality%
\footnote{By cardinality we mean the number of nonzero components.%
} of each row of $F$ and $R$ is at least one, and the cardinality
of each column of $R-F$ is at least two, usually three. Therefore,
$R-F$ may be viewed as the incidence matrix of a directed hypergraph.
The matrices $F$ and $R$ are sparse and the particular sparsity
pattern depends on the particular biochemical network being modelled.

Let $u\in\mathbb{R}_{>0}^{m}$ denote a variable vector of molecular
species concentrations. Assuming constant nonnegative elementary kinetic
parameters $k_{f},k_{r}\in\mathbb{R}_{\geq0}^{n}$, we presume \emph{elementary
reaction kinetics} for forward and reverse elementary reaction rates
as $s(k_{f},u):=\exp(\ln(k_{f})+F^{T}\ln(u))$ and $r(k_{r},u):=\exp(\ln(k_{r})+R^{T}\ln(u))$,
respectively, where $\exp(\cdot)$ and $\ln(\cdot)$ denote the respective
componentwise functions. Then, the deterministic dynamical equation
for time evolution of molecular species concentration is given by
\begin{eqnarray}
\frac{du}{dt} & \equiv & (R-F)(s(k_{f},u)-r(k_{r},u))\label{eq:dcdt1}\\
 & = & (R-F)(\exp(\ln(k_{f})+F^{T}\ln(u))-\exp(\ln(k_{r})+R^{T}\ln(u))).\label{eq:dcdt2}
\end{eqnarray}
Investigation of steady states plays a crucial role in the modelling
of biochemical reaction systems. If one transforms \eqref{eq:dcdt2}
to logarithmic scale, by letting $w\equiv[\ln(k_{f})^{T},\,\ln(k_{r})^{T}]^{T}\in\mathbb{R}^{2n}$, $x\equiv\ln(u)\in\mathbb{R}^{m}$,
then, up to a sign, the right-hand side of \eqref{eq:dcdt2} is equal
to the function
\begin{equation}
f(x):=([F,\, R]-[R,\, F])\exp\left(w+[F,\, R]^{T}x\right),\label{eq:f(x)}
\end{equation}
where $\left[\,\cdot\thinspace,\cdot\,\right]$ stands for the horizontal
concatenation operator. Thus, we shall focus on finding the points
$x\in\mathbb{R}^{m}$ such that $f(x)=0_m$, which correspond to the
steady states of the dynamical equation~\eqref{eq:dcdt1}.

A point $\bar{x}$ will be a zero of the function $f$ if and only
if $\|f(\bar{x})\|^{2}=0$. Denoting
\begin{align*}
p(x) & :=[F,R]\exp\left(w+[F,R]^{T}x\right),\\
c(x) & :=[R,F]\exp\left(w+[F,R]^{T}x\right),
\end{align*}
one obtains, as in Example~\ref{ex:Finding-zeroes},
\[
\|f(x)\|^{2}=\|p(x)-c(x)\|^{\text{2}}=2\left(\|p(x)\|^{\text{2}}+\|c(x)\|^{\text{2}}\right)-\|p(x)+c(x)\|^{\text{2}}.
\]
Again, as all the components of $p(x)$ and $c(x)$ are positive and
convex functions%
\footnote{Note that $p(x)$ is the rate of production of each molecule and $c(x)$
is the rate of consumption of each molecule.%
}, both
\begin{equation}\label{eq:f1andf2}
f_{1}(x):=2\left(\|p(x)\|^{\text{2}}+\|c(x)\|^{\text{2}}\right)\quad\text{
and} \quad f_{2}(x):=\|p(x)+c(x)\|^{\text{2}}
\end{equation} are convex functions. In
addition to this, both $f_{1}$ and $f_{2}$ are smooth, having
\begin{eqnarray*}
\nabla f_{1}(x) & = & 4\nabla p(x)p(x)+4\nabla c(x)c(x),\\
\nabla f_{2}(x) & = & 2\left(\nabla p(x)+\nabla c(x)\right)\left(p(x)+c(x)\right),
\end{eqnarray*}
see e.g.~\cite[pp.~245--246]{nocedal_numerical_2006}, with
\begin{eqnarray*}
\nabla p(x) & = & [F,R]\textrm{EXP}\left(w+[F,R]^{T}x\right)[F,R]^{T},\\
\nabla c(x) & = & [F,R]\textrm{EXP}\left(w+[F,R]^{T}x\right)[R,F]^{T},
\end{eqnarray*}
where $\text{EXP}\left(\cdot\right)$ denotes the diagonal matrix
whose entries are the elements in the vector $\exp\left(\cdot\right)$.

Setting $\phi(x):=f_{1}(x)-f_{2}(x)$, the problem of finding a zero
of $f$ is equivalent to the following optimisation problem:
\begin{equation}
\ensuremath{\underset{x\in\mathbb{R}^{m}}{\textrm{minimise}}}\;\phi(x):=f_{1}(x)-f_{2}(x).\label{eq:DC1-1}
\end{equation}
 We now prove that $\phi$ satisfies the \L{}ojasiewicz property.
Denoting $A:=[F,\, R]-[R,\, F]$ and $B:=[F,\, R]^{T}$ we can write
\begin{align*}
\phi(x) & =f(x)^{T}f(x)=\exp\left(w+Bx\right)^{T}A^{T}A\exp\left(w+Bx\right)^{T}\\
 & =\exp\left(w+Bx\right)^{T}Q\exp\left(w+Bx\right)^{T}\\
 &=\sum_{j,k=1}^{2n}q_{j,k}\exp\left(w_{j}+w_{k}+\sum_{i=1}^{m}(b_{ji}+b_{ki})x_{i}\right),
\end{align*}
where $Q=A^{T}A.$ Since~$b_{ij}$ are nonnegative integers for all
$i$ and~$j$, we conclude that the function~$\phi$ is real analytic
(see Proposition~2.2.2 and Proposition~2.2.8 in~\cite{parks1992primer}).
 It follows from Proposition \ref{prop:real_analytic} that the function
$\phi$ satisfies the \L{}ojasiewicz property with some exponent
$\theta\in[0,1)$.

Finally, as in Example~\ref{ex:Finding-zeroes}, for all $\rho>0$,
the function $g(x):=f_{1}(x)+\frac{\rho}{2}\|x\|^{2}$ is twice continuously
differentiable, which implies that $\nabla g$ is locally Lipschitz
continuous. Therefore, either Theorem~\ref{th:convergence_Lowasiewicz}
or Corollary~\ref{cor:convergence_Loj} guarantee the convergence
of the sequence generated by BDCA, as long as the sequence is bounded.
\begin{remark}
In principle, one cannot guarantee the linear convergence of BDCA
applied to biochemical problems for finding steady states. Due to
the \emph{mass conservation assumption}, $\exists l\in\mathbb{R}_{>0}^{m}$
such that $(R-F)^{T}l=0_{n}$. This implies that~$\nabla f(x)$ is
singular for every~$x\in\mathbb{R}^{m}$, because
\[
\nabla f(x)l=[F,R]\textrm{EXP}\left(w+[F,R]^{T}x\right)[F-R,R-F]^{T}l=0_{m}.
\]
Therefore, the reasoning in Example~\ref{ex:Finding-zeroes} cannot
be applied. However, one can still guarantee that any stationary point of
$ \|f(x)\|^{2}$ is actually a steady state  of the considered biochemical reaction network if the function $f$ is strictly duplomonotone~\cite{artacho2015globally}. A function $f:\mathbb{R}^{m}\to\mathbb{R}^{m}$ is called  \emph{duplomonotone} with constant $\bar{\tau}>0$
if
$$
\langle f(x)-f(x-\tau f(x)),f(x)\rangle\geq0\quad\text{whenever }x\in\mathbb{R}^{m},0\leq\tau\leq\bar{\tau},
$$
and  \emph{strictly duplomonotone} if this inequality is strict whenever $f(x)\neq0_m$. If $f$ is differentiable and strictly duplomonotone then
$\nabla f(x)f(x)=0_m$ implies $f(x)=0_m$~\cite{artacho2015globally}. We previously established that some stoichiometric matrices do give rise to strictly duplomonotone functions~\cite{artacho2015globally}, and our numerical experiments, described next, do support the hypothesis that this is a pervasive property of many biochemical networks.
\end{remark}

\section{Numerical Experiments\label{sec:Application-to-biochemistry-1}}

The codes are written in MATLAB and the experiments were performed
in MATLAB version R2014b on a desktop Intel Core i7-4770 CPU @3.40GHz
with 16GB RAM, under Windows 8.1 64-bit. The subproblems~$(\mathcal{P}_{k})$
were approximately solved using the function\texttt{ fminunc }with
{\texttt{optimoptions('fminunc', \hfill 'Algorithm',\hfill 'trust-region', \hfill 'GradObj', \hfill
'on', \hfill 'Hessian', 'on', 'Display', 'off',  'TolFun', 1e-8, 'TolX', 1e-8)}}.

In Table~\ref{table:results} we report the numerical results comparing
DCA and BDCA with quadratic interpolation (Algorithm~3) for 14 models
arising from the study of systems of biochemical reactions. The parameters used were
$\alpha=0.4$, $\beta=0.5$, $\bar{\lambda}=50$ and $\rho=100$. We only
provide the numerical results for Algorithm~3 because it normally
gives better results than Algorithm~2 for biochemical models, as
it is shown in Figure~\ref{fig:comparison_DCA_BDCA_2and3}. In Figure~\ref{fig:Comparison_DCA_BDCA}
we show a comparison of the rate of convergence of DCA and BDCA with
quadratic interpolation for two big models. In principle, a relatively large value of the parameter $\rho$ could slow down the convergence of DCA. This is not the case here: the behaviour of DCA is usually the same for  values of $\rho$ between $0$ and $100$, see Figures~\ref{fig:comparison_DCA_BDCA_2and3} and Figure~\ref{fig:Comparison_DCA_BDCA} (left). In fact, for big models, we observed that  a value of $\rho$ between $50$ and $100$ normally accelerates the convergence of both DCA and BDCA, as shown in Figure~\ref{fig:Comparison_DCA_BDCA} (right). For these reasons, for the numerical results in Table~\ref{table:results}, we applied both DCA and BDCA  to the regularized version $g(x)-h(x)$ with $g(x) =f_1 (x)+\frac{100}{2}\|x\|^2$ and $h(x)=f_2(x)+\frac{100}{2}\|x\|^2$, where $f_1$ and $f_2$ are given by~\eqref{eq:f1andf2}.

\begin{figure}[ht!]
\begin{centering}
\includegraphics[width=0.8\textwidth]{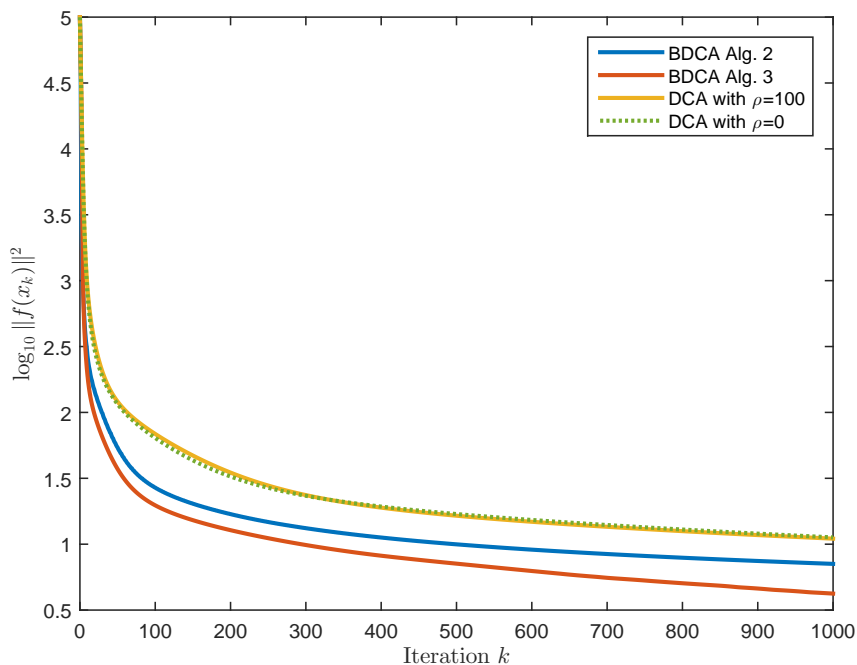}
\par\end{centering}

\protect\caption{Comparison of the rate of convergence of DCA (Algorithm 1)  with $\rho\in\{0,100\}$ and BDCA
(Algorithms 2 and 3) with $\rho=100$ for finding a steady state of the ``Ecoli\_core''
model ($m=72$, $n=94$). \label{fig:comparison_DCA_BDCA_2and3}}
\end{figure}

\begin{figure}[ht!]
\begin{centering}
\includegraphics[width=0.49\textwidth]{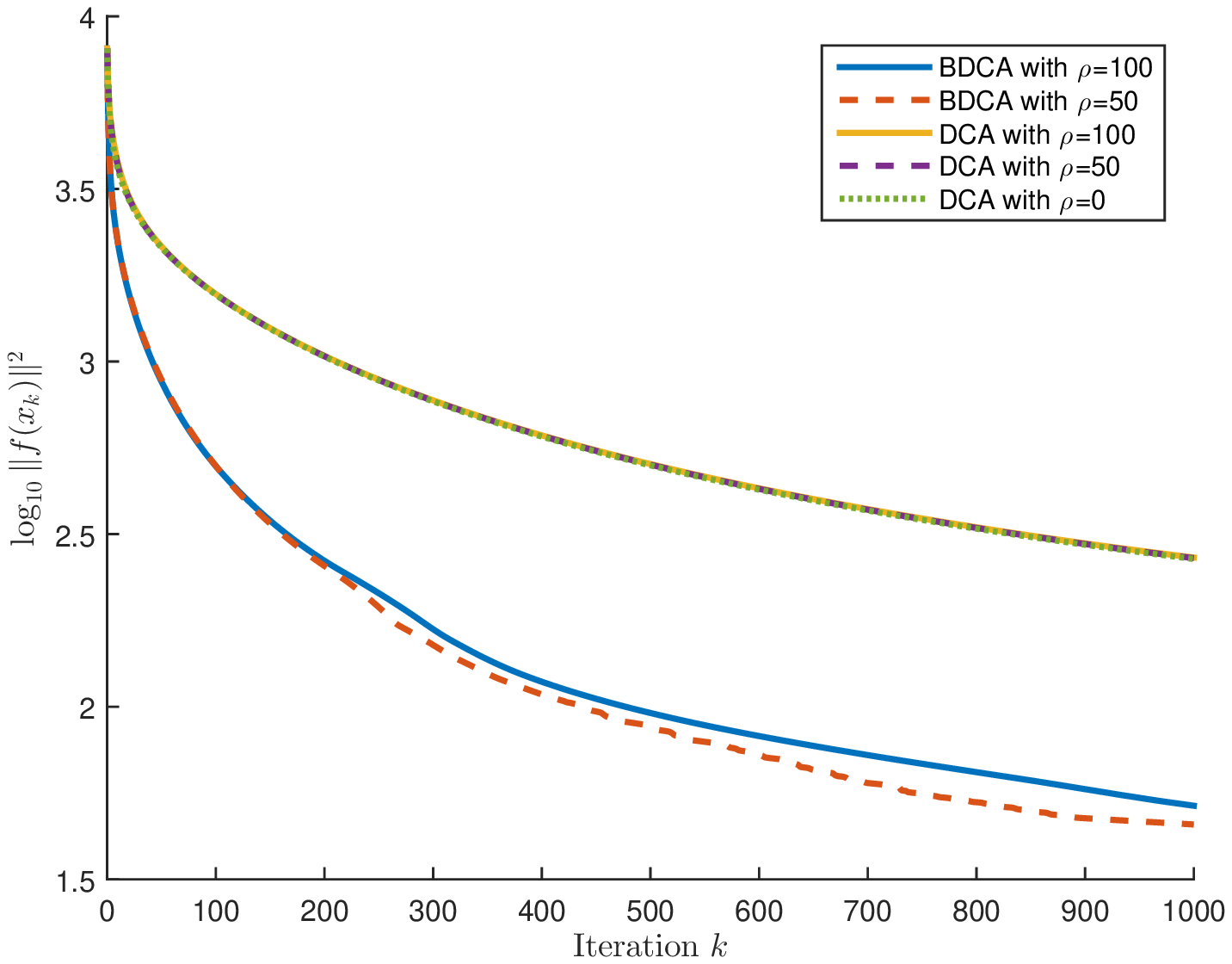}\includegraphics[width=0.49\textwidth]{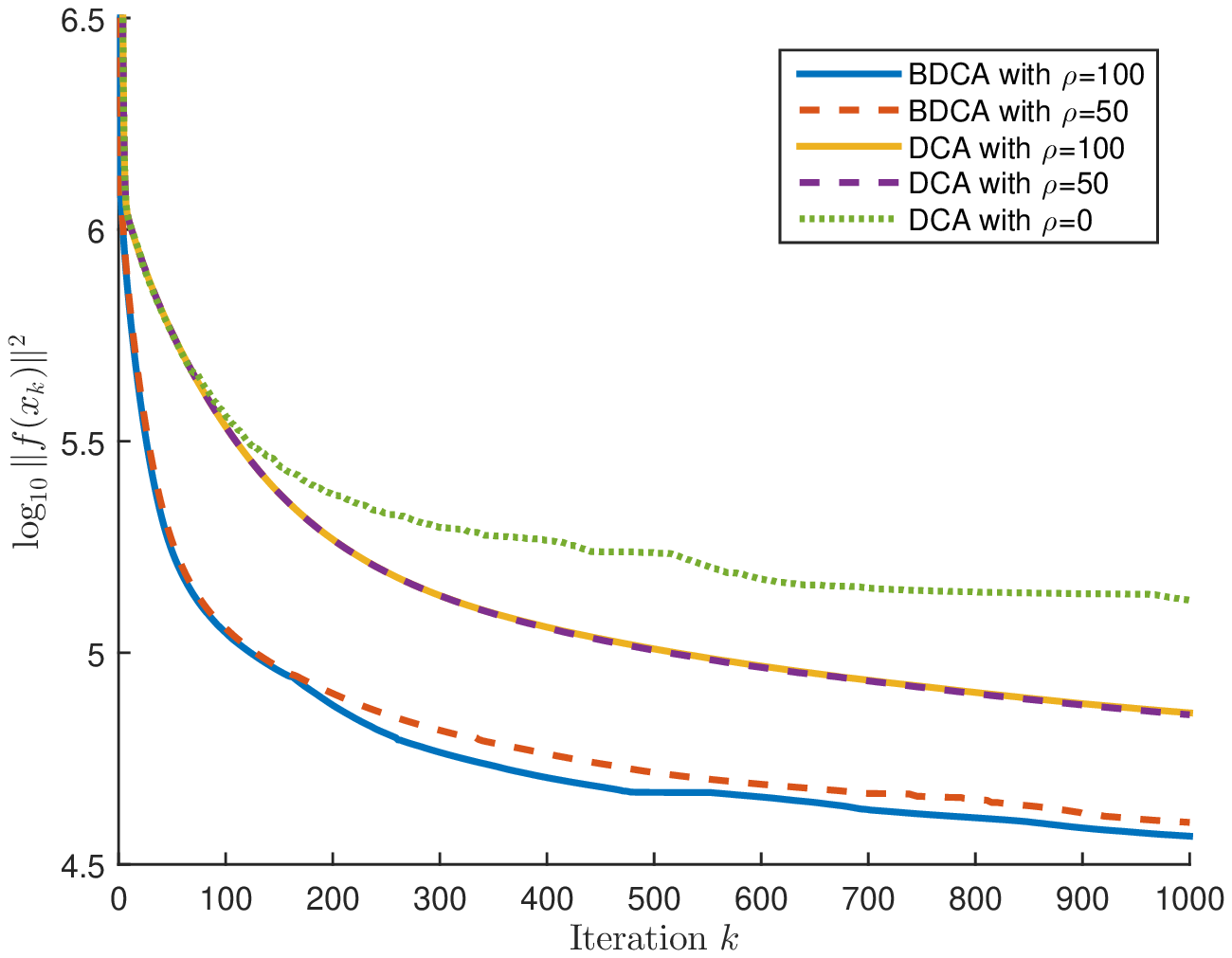}
\par\end{centering}

\protect\caption{Comparison of the rate of convergence of DCA and BDCA (Algorithm~3)
for finding a steady state for different values of the parameter $\rho$. On the left we show the ``iJO1366''
model ($m=1655$, $n=2416$). On the right we show the human metabolism
model ``Recon205\_20150128'' ($m=4085$, $n=7400$)~\cite{Thiele2013}.
\label{fig:Comparison_DCA_BDCA}}
\end{figure}

\section{Concluding Remarks}

In this paper, we introduce two new algorithms for minimising smooth
DC functions, which we term \emph{Boosted Difference of Convex function
Algorithms} (BDCA). Our algorithms combine DCA together with a line
search, which utilises the point generated by DCA to define a search
direction. This direction is also employed by Fukushima--Mine in~\cite{fukushima_generalized_1981},
with the difference that our algorithms start searching for the new
candidate from the point generated by DCA, instead of starting from
the previous iteration. Thus, our main contribution comes from the
observation that this direction is not only a descent direction for
the objective function at the previous iteration, as observed by Fukushima--Mine,
but is also a descent direction at the point defined by DCA. Therefore,
with the slight additional computational effort of a line search one
can achieve a significant decrease in the value of the function.  This result cannot
be directly generalized for nonsmooth functions, as shown in Remark~\ref{rem:nonsmooth}. We
prove that every cluster point of the algorithms are stationary points
of the optimisation problem. Moreover, when the objective function
satisfies the \L{}ojasiewicz property, we prove global convergence
of the algorithms and establish convergence rates.

We demonstrate that the important problem of finding a steady state
in the dynamical modelling of systems of biochemical reactions can
be formulated as an optimisation problem involving a difference of
convex functions. We have performed numerical experiments, using models
of systems of biochemical reactions from various species, in order
to find steady states. The tests clearly show that our algorithm outperforms
DCA, being able to achieve the same decrease in the value of the DC
function while employing substantially less iterations and time. On
average, DCA needed five times more iterations to achieve the same
accuracy as BDCA. Furthermore, our implementation of BDCA was also
more than four times faster than DCA. In fact, the slowest instance
of BDCA was always at least three times faster than DCA. This substantial
increase in the performance of the algorithms is especially relevant
when the typical size of the problems is big, as is the case with
all realistic biochemical network models.

\begin{centering}
\begin{sidewaystable}
\begin{centering}
\footnotesize
\begin{tabular}{|l|r|r||r|r||r|r|r||r|r|r|r|r|r||r|r|}
\hline
\multicolumn{3}{|c||}{{DATA}} & \multicolumn{2}{c||}{{INSTANCES}} & \multicolumn{3}{c||}{{BDCA}} & \multicolumn{6}{c||}{{DCA}} & \multicolumn{2}{c|}{{RATIO (avg.)}}\tabularnewline
\cline{1-14}
\multirow{2}{*}{{Model Name}} & \multirow{2}{*}{{m}} & \multirow{2}{*}{{n}} & {avg.} & {avg.} & \multicolumn{3}{c||}{{time (seconds)}} & \multicolumn{3}{c|}{{iterations}} & \multicolumn{3}{c||}{{time (seconds)}} & \multicolumn{2}{c|}{{DCA/BDCA}}\tabularnewline
\cline{6-16}
 &  &  & {$\phi(x_{0})$} & {$\phi(x_{{\rm end}})$} & {min.} & {max.} & {avg.} & {min.} & {max.} & {avg.} & {min.} & {max.} & {avg.} & {~iter.~} & {time}\tabularnewline
\hline
\hline
{Ecoli core} & {72} & {94} & {5.28e6} & {5.80} & {16.6} & {25.6} & {19.7} & {4101} & {6176} & {4861} & {68.0} & {104.7} & {86.8} & {4.9} & {4.4}\tabularnewline
\hline
{L lactis MG1363} & {486} & {615} & {2.00e7} & {62.73} & {2926.1} & {4029.1} & {3424.4} & {4164} & {6362} & {5241} & {14522.4} & {18212.5} & {16670.1} & {5.2} & {4.9}\tabularnewline
\hline
{Sc thermophilis } & {349} & {444} & {1.95e7} & {84.99} & {290.8} & {552.7} & {358.4} & {4021} & {6303} & {4873} & {1302.1} & {2003.8} & {1611.1} & {4.9} & {4.5}\tabularnewline
\hline
{T Maritima} & {434} & {554} & {3.54e7} & {114.26} & {1333.0} & {2623.3} & {1919.7} & {3536} & {5839} & {4700} & {5476.2} & {12559.2} & {8517.1} & {4.7} & {4.4}\tabularnewline
\hline
{iAF692} & {466} & {546} & {2.32e7} & {57.42} & {1676.8} & {2275.3} & {1967.4} & {4215} & {7069} & {5303} & {8337.0} & {11187.5} & {9466.3} & {5.3} & {4.8}\tabularnewline
\hline
{iAI549} & {307} & {355} & {1.10e7} & {35.90} & {177.2} & {254.4} & {209.2} & {3670} & {5498} & {4859} & {665.1} & {1078.2} & {913.4} & {4.9} & {4.4}\tabularnewline
\hline
{iAN840m} & {549} & {840} & {2.58e7} & {105.18} & {3229.1} & {6939.3} & {4720.6} & {4254} & {5957} & {4971} & {16473.3} & {28956.7} & {21413.2} & {5.0} & {4.5}\tabularnewline
\hline
{iCB925} & {416} & {584} & {1.52e7} & {67.54} & {1830.7} & {2450.5} & {2133.4} & {3847} & {6204} & {5030} & {7358.2} & {11464.6} & {9886.8} & {5.0} & {4.6}\tabularnewline
\hline
{iIT341} & {425} & {504} & {7.23e6} & {139.71} & {1925.2} & {2883.1} & {2301.8} & {3964} & {9794} & {5712} & {9433.8} & {20310.3} & {12262.0} & {5.7} & {5.3}\tabularnewline
\hline
{iJR904} & {597} & {915} & {1.47e7} & {139.63} & {6363.1} & {9836.2} & {7623.0} & {4173} & {5341} & {4776} & {24988.5} & {43639.8} & {33620.6} & {4.4} & {4.8}\tabularnewline
\hline
{iMB745} & {528} & {652} & {2.77e7} & {305.80} & {2629.1} & {5090.7} & {4252.3} & {3986} & {7340} & {5020} & {16437.8} & {25171.6} & {20269.3} & {5.0} & {4.8}\tabularnewline
\hline
{iSB619} & {462} & {598} & {1.64e7} & {40.64} & {2406.7} & {5972.2} & {3323.5} & {2476} & {6064} & {4260} & {8346.1} & {25468.1} & {13966.9} & {4.3} & {4.2}\tabularnewline
\hline
{iTH366} & {587} & {713} & {3.42e7} & {63.37} & {3310.2} & {5707.3} & {4464.2} & {4089} & {6363} & {4965} & {13612.7} & {30044.1} & {20715.5} & {5.0} & {4.6}\tabularnewline
\hline
{iTZ479 v2} & {435} & {560} & {1.97e7} & {78.12} & {1211.4} & {2655.8} & {2216.4} & {3763} & {6181} & {4857} & {7368.1} & {12591.6} & {10119.8} & {4.9} & {4.6}\tabularnewline
\hline
\end{tabular}
\par\end{centering}

\protect\caption{Performance comparison of BDCA and DCA for finding a steady state
of various biochemical reaction network models. For each model, we
selected a random kinetic parameter $w\in[-1,1]^{2n}$ and we randomly
chose 10 initial points~$x_{0}\in[-2,2]^{m}$. For each~$x_{0}$,
BDCA was run 1000 iterations, while DCA was run until it reached the
same value of $\phi(x)$ as obtained with BDCA.\label{table:results}}
\end{sidewaystable}
\end{centering}

\section*{Acknowledgements}

The authors are grateful to the anonymous referee for their pertinent and
constructive comments. The authors also thank Aris Daniilidis for his helpful information
on the \L{}ojasiewicz exponent.\\

\noindent F.J.  Arag\'on Artacho was supported by MINECO of Spain  and  ERDF of EU, as part of the
Ram\'on y Cajal program (RYC-2013-13327) and the grant MTM2014-59179-C2-1-P.
R.M. Fleming and P.T. Vuong were supported by the U.S. Department
of Energy, Offices of Advanced Scientific Computing Research and the
Biological and Environmental Research as part of the Scientific Discovery
Through Advanced Computing program, grant \#DE-SC0010429.

\bibliographystyle{plain}

\end{document}